\documentclass[12pt,a4paper]{article}
\usepackage{}
\usepackage{indentfirst}
\usepackage{amsfonts}
\usepackage{amssymb}
\usepackage{mathrsfs}
\usepackage{amsmath}
\usepackage{amsthm}
\usepackage{enumerate}
\usepackage{cite}
\allowdisplaybreaks
\newtheorem{defn}{Definition}[section]
\newtheorem{thm}[defn]{Theorem}
\newtheorem{lem}[defn]{Lemma}
\newtheorem{prop}[defn]{Proposition}

\newtheorem{re}[defn]{Remark}

\newcommand{\ad}{{\rm ad}}
\newcommand{\ch}{{\rm ch}}
\newcommand{\F}{\textbf{F}}
\newcommand{\LL}{\mathcal{L}}
\newcommand{\mF}{\mathcal{F}}

\newcommand{\R}{\mathbb{R}}
\newcommand{\C}{\mathbb{C}}
\newcommand{\N}{\mathbb{N}}
\newcommand{\Z}{\mathbb{Z}}
\newcommand{\supercite}[1]{\textsuperscript{\cite{#1}}}
\newcommand{\ba}{\begin{array}}
\newcommand{\ea}{\end{array}}
\newcommand{\bt}{\begin{tabular}}
\newcommand{\et}{\end{tabular}}
\newcommand{\btb}{\begin{table}}
\newcommand{\etb}{\end{table}}
\newcommand{\bc}{\begin{center}}
\newcommand{\ec}{\end{center}}
\newcommand{\bea}{\begin{eqnarray}}
\newcommand{\eea}{\end{eqnarray}}
\newcommand{\Bea}{\begin{eqnarray*}}
\newcommand{\Eea}{\end{eqnarray*}}
\newcommand{\beq}{\begin{equation}}
\newcommand{\eeq}{\end{equation}}
\newcommand{\Beq}{\begin{equation*}}
\newcommand{\Eeq}{\end{equation*}}

\begin{document}
\title{{ \bf Systems of quotients of Lie triple systems}
   \author{{ Yao Ma$^1$, Liangyun Chen$^1$,  Jie Lin$^2$ }
   \\{}$^1$ School of Mathematics and Statistics, Northeast Normal University,
  \\ Changchun,  130024,  CHINA
  \\{}$^2$ Sino-European Institute of Aviation Engineering,\\ Civil Aviation University of China,  Tianjin, 300300,  CHINA}}
 \date{}

\maketitle

\begin{abstract}
In this paper, we introduce the notion of system of quotients of Lie triple systems and investigate some properties which
can be lifted from a Lie triple system to its systems of quotients. We relate the notion of Lie triple system of Martindale-like quotients
with respect to a filter of ideals and the notion of system of quotients, and prove that the system of quotients of a Lie triple system is
equivalent to the algebra of quotients of a Lie algebra in some sense, and these allow us to construct the maximal system of quotients for
 nondegenerate Lie triple systems.

\bigskip

\end{abstract}
\renewcommand{\thefootnote}{\fnsymbol{footnote}}
\footnote[0]{ Address correspondence to Prof. Liangyun Chen,  School
  of Mathematics and Statistics, Northeast Normal University,
  Changchun 130024, China; E-mail: chenly640@nenu.edu.cn.}

\section{Introduction}

Lie triple systems arose initially in Cartan's study of Riemannian geometry, but whose concept was introduced by Nathan Jacobson in 1949 to study
subspaces of associative algebras closed under triple commutators $[[u,v],w]$(cf. \cite{J1}). The role played by Lie triple systems in  the theory of symmetric
 spaces is parallel to that of Lie algebras in the theory of Lie groups: the tangent space at every point of a symmetric space has the structure of a Lie triple system.

The notion of ring of quotients was introduced by Utumi in 1956 (cf. \cite{U}). He proved that the ring without right zero divisors has a maximal left quotient ring
 and constructed it.  Inspired by \cite{U}, Siles Molina studied the algebras of quotients of Lie algebras (cf. \cite{S}). The notion of Martindale ring of quotients
 was introduced by Martindale in 1969 for prime rings (cf. \cite{M2}). In \cite{GG}, E. Garc\'{i}a and M. G\'{o}mez defined Martindale-like quotients for Lie triple
 systems with respect to power filters of sturdy ideals and constructed the maximal system in the nondegenerate cases.

In this paper we introduce the notion of system of quotients of Lie triple systems and prove that some properties such as semiprimeness, primeness
or nondegeneracy can be lifted from a Lie triple system to its systems of quotients. We answer the question about the relation between $S$ being
a Lie triple system of Martindale-like quotients with respect to a filter and $S$ being a system of quotients of a Lie triple system $T$. We also prove that if $S$ is
 a system of quotients of a semiprime Lie triple system $T$, then
$L(S)=S\oplus\LL(S,S)$ is an algebra of quotients of $L(T)=T\oplus \LL(T,T)$ in Theorem \ref{S/T implies L(S)/L(T)}. Finally, we construct the maximal system of
quotients for a nondegenerate Lie triple system, and show that the maximal system of quotients of a finite dimensional semisimple Lie triple system over an algebraically closed field of characteristic 0 is itself.

Throughout this paper, we let $\F$   be  a field of arbitrary
characteristic. For background material on Lie triple systems the
reader is referred to \cite{GG,GGN,J,J1,L}.  Our notation and terminology
are standard as may be found in \cite{G,S,S1}.

\section{Preliminaries}

\begin{defn}{\rm \supercite{M}}
A vector space $T$ together with a trilinear map $(x, y, z)\mapsto[x,y,z]$ is called a Lie triple system(LTS) if
\begin{enumerate}[(1)]
\item $[x,x,z]=0$,
\item $[x,y,z]+[y,z,x]+[z,x,y]=0$,
\item $[u,v,[x,y,z]]=[[u,v,x],y,z]+[x,[u,v,y],z]+[x,y,[u,v,z]]$,
\end{enumerate}
for all $x,y,z,u,v\in T$.
\end{defn}

\begin{defn}{\rm\supercite{L}}
A subsystem of a LTS $T$ is a subspace $I$ for which $[I,I,I]\subseteq I$. An ideal of a LTS $T$ is a subspace $I$ for which $[I,T,T]\subseteq I$,
in this case we have that $[T,I,T]$ and $[T,T,I]$ are contained in $I$.
\end{defn}

If $L$ is a Lie algebra, then $L$ is a LTS relative to $[a,b,c]\equiv[[a,b],c]$. Conversely, it was showed in \cite{CF} that if $T$ is a LTS, then the standard embedding of $T$ is the $\Z_2$-graded Lie algebra $L(T) =
L_0 \oplus L_1$, $L_0$ being the $\F$-span of $\{\LL(x, y): x, y \in T \}$, denoted by $\LL(T,T)$, where $\LL(x, y)$ denotes the left
multiplication operator in $T$, $\LL(x, y)(z) := [x, y, z]$; $L_1 := T$ and where the product is given by
$$[\left(\LL(x, y), z\right), \left(\LL(u, v),w\right)]:=\left(\LL([u, v, y], x) - \LL([u, v, x], y) + \LL(z,w), [x, y, w] - [u, v, z]\right).$$
Let us observe that $L_0$ with the product induced by the one in $L(T) = L_0 \oplus L_1$ becomes a Lie algebra. Moreover, for a $\Z_2$-graded Lie algebra $L=L_0\oplus L_1$, $L_1$
 has the structure of a LTS and every LTS $T$ is the $1$ component of a $\Z_2$-graded Lie algebra since its standard imbedding $L(T)$ is a
  $\Z_2$-graded Lie algebra with $L(T)_0=\LL(T,T), L(T)_1=T$.

\begin{defn}{\rm\supercite{S}}
A Lie algebra $L$ is said to be semiprime if $[I, I]\neq0$ for every nonzero ideal $I$ of $L$. $L$ is said to be prime if every two nonzero ideals $I, J$
of $L$ give $[I, J]\neq0$. $L$ is said to be nondegenerate if $[[L, x], x]\neq0$ for every nonzero element $x$ of $L$. For a nonempty subset $I$ of $L$ the
 set $C_L(I)=\{x\in L|[x,I]=0\}$ is called the centralizer of $I$ in $L$.
\end{defn}

\begin{defn}{\rm\supercite{CF}}
A LTS $T$ is semiprime if $[T, I, I]\neq0$ for every nonzero ideal $I$ of $T$. It is prime if every two nonzero ideals $I, J$ of $T$ give $[T, I, J]\neq0$.
 $T$ is called nondegenerate if $[T, x, x]\neq 0$ for every nonzero element $x$ of $T$. Let $I$ be a nonempty subset of $T$. $C_{T}(I)=\{x\in T|[x,I,T]=[T,I,x]=0\}$ is
 called the centralizer of $I$ in $T$. In particular, $C_{T}(T)=\{x\in T|[x,T,T]=0\}$ is called the center of $T$, and is denoted by $C(T)$.
\end{defn}

\begin{defn}
An ideal $I$ of a LTS $T$ is said to be essential if every nonzero ideal of $T$ hits $I$ ($I\cap J\neq 0$ for every nonzero ideal $J$ of $T$).
\end{defn}

Clearly, if $I$ and $J$ are essential ideals, then $I\cap J$ is an
essential ideal.

The following results contain analogous results to the corresponding
ones for Lie algebras in \cite{S}, and their proofs are similar.
Note that we will always consider a LTS as the $1$ component
 of its standard imbedding.

\begin{prop}{\rm\supercite{CF}}\label{center and essential}
Let $I$ be an ideal of a LTS $T$. Then
\begin{enumerate}[(1)]
\item $C_{T}(I)$ is an ideal of $T$.
\item\label{center implies essential} If~~$C_{T}(I)=0$, then $I$ is essential. Moreover, if $T$ is semiprime, then $I\cap C_{T}(I)=0$ and $I$
 is essential if and only if $C_T(I)=0$.
\end{enumerate}
\end{prop}

\begin{prop}\label{converse not true}
If $T$ is a LTS, then
\begin{center}
$T$ is nondegenerate $\Rightarrow$ $T$ is semiprime $\Rightarrow$ $C(T)=0$.
\end{center}
\end{prop}

However, the converses of Proposition \ref{converse not true} are not true. Inspired by the counterexample in \cite{S}, we consider a vector space
$A=\left\{\left( \begin{array}{cc}
a & b \\
0 & 0
\end{array} \right)
| ~a,b\in\R\right\}$. Then $A$ becomes a LTS $A^{-}$ by putting
$[E,F,G]=EFG-FEG-GEF+GFE, \forall E, F, G\in A$. It is easy to prove  $C(A^{-})=0$. Note that
$I^-=\left\{\left( \begin{array}{cc}
0 & b \\
0 & 0
\end{array} \right)
|~ b\in\R\right\}$ is an ideal of $A^{-}$ which satisfies $[A^{-}, I^-, I^-]=0$, hence $A^{-}$ is not semiprime.

In \cite{G}, the author constructed a semiprime degenerate Lie algebra, and we denote it by $L$, then $L$ can be regarded as a LTS $\tilde{L}$ with $[a,b,c]=[[a,b],c], \forall a,b,c\in L$. Then $\tilde{L}$ is a semiprime degenerate LTS.

\begin{re}
A LTS $T$ is prime if $C_{T}(I)=0$ for every nonzero ideal $I$ of $T$.
\end{re}

\section{Systems of quotients of a Lie triple system}

Inspired by the notion of algebra of quotients of Lie algebras in \cite{S}, we introduce the notion of system of quotients of Lie triple systems.

Suppose that $T$ and $S$ are two LTS such that $T\subseteq S$, and $R(-,-):S\rightarrow S$ is defined by $R(x,y)(z)=[z,x,y], \forall x,y,z\in S$. For
every $s\in S$, set
$$_T(s)=\F s+\left\{\sum_{i=1}^n R(x^i_{1},y^i_{1}) \cdots R(x^i_{k_i},y^i_{k_i})(s)
\ \vert\ x^i_j,y^i_j\in T \text{ with } n, k_i\in \mathbb{N} \right\}.$$
That is, $_T(s)$ is the linear span in $S$ of the elements of the
form $R(x_1,y_1)\cdots R(x_n,y_n)(s)$ and $s$, where $n\in\N$ and
$x_1, \cdots, x_n, y_1, \cdots, y_n\in T$. It is clear that $[{_T}(s), T, T]\subseteq {_T}(s)$, $[T, {_T}(s), T]\subseteq {_T}(s)$ and  $[ T, T, {_T}(s)]\subseteq {_T}(s)$. In particular, in the case that $s\in T$, $_T(s)$ is the
ideal of $T$ generated by $s$. Moreover, we define
$$(T:s)=\{x\in T|~[x,T,{ _T(s)}]+[x,{ _T(s)},T]\subseteq T\}.$$
Obviously, if $s\in T$, then $(T:s)=T$. By definition, for $x\in (T:s)$,  $[x,T, {_T}(s)]$ and  $[x,{_T}(s), T]$ are contained in $T$, moreover, by identity (1) in the definition of Lie triple systems,  $[T,x, {_T}(s)]$ and  $[{_T}(s),x, T]$ are contained in $T$. Finally by (2),
$[T,{_T}(s), x] \subseteq [{_T}(s), x,T] +[x, T, {_T}(s)] \subseteq T$.

\begin{prop}\label{(T:s) is an ideal}
Let $T$ be a subsystem of a LTS $S$ and take $s\in S$. Then $(T:s)$
is an ideal of $T$. Moreover, it is maximal among the ideal $I$ of  $T$ such that $[I,T,s]+[I,s,T]\subseteq T$.
\end{prop}
\begin{proof} It is clear that $(T:s)$ is a subspace of $T$. Now, for any
  $x\in (T:s)$, we have \Bea
[T,{_{T}(s)},[x,T,T]]&\subseteq&[[T,{_{T}(s)},x],T,T]+[x,[T,{_{T}(s)},T],T]+[x,T,[T,{_{T}(s)},T]]\\
&\subseteq& [T,T,T]+[x,{_{T}(s)},T]+[x,T,{_{T}(s)}]\subseteq T
\Eea
and
\Bea
[[x,T,T],{_{T}(s)},T]&\subseteq&[x,T,[T,{_{T}(s)},T]]+[T,[x,T,{_{T}(s)}],T]+[T,{_{T}(s)},[x,T,T]]\\
&\subseteq&[x,T,{_{T}(s)}]+[T,T,T]+T\subseteq T.
\Eea
It follows that $[[x,T,T],T,{_{T}(s)}]\subseteq [T,{_{T}(s)},[x,T,T]]+[[x,T,T],{_{T}(s)},T]\subseteq T$, and so $[x,T,T]\subseteq (T:s)$, i.e., $(T:s)$ is an ideal of $T$.

Suppose that $I$ is an ideal of $T$  such that $[I,T,s]+[I,s,T]\subseteq T$. We will show that for all $x_1, \cdots, x_n, y_1, \cdots, y_n\in T, n\in \N$,
$$[I,T,R(x_{1},y_{1})\cdots R(x_{n},y_{n})(s)]\subseteq T, ~~~[I,R(x_{1},y_{1})\cdots R(x_{n},y_{n})(s),T]\subseteq T.$$
We prove it by induction on $n\geq1$. The base step holds since
\Bea
[I,T,[s,x_{1},y_{1}]]&\subseteq&[[I,T,s],x_{1},y_{1}]+[s,[I,T,x_{1}],y_{1}]+[s,x_{1},[I,T,y_{1}]]\\
&\subseteq &[T,T,T]+[s,I,T]+[s,T,I]\subseteq T
\Eea
and
\Bea
[I,[s,x_{1},y_{1}],T]&\subseteq&[s,x_{1},[I,y_{1},T]]+[[s,x_{1},I],y_{1},T]+[I,y_{1},[s,x_{1},T]]\\
&\subseteq & [s,T,I]+[T,T,T]+T\subseteq T.
\Eea

For the inductive step, assume $$[I,T,R(x_{2},y_{2})\cdots R(x_{n},y_{n})(s)]\subseteq T \text{ and } [I,R(x_{2},y_{2})\cdots R(x_{n},y_{n})(s),T]\subseteq T.$$
Apply step one to the element $R(x_{2},y_{2})\cdots R(x_{n},y_{n})(s)$ we have $[I,T,_{T}(s)]+[I,_{T}(s),T]\subseteq T$, and so $I\subseteq (T:s)$.
\end{proof}

\begin{defn}
Let $T$ be a subsystem of a LTS $S$. Then $S$ is called a system of quotients of $T$, if given $s,s'\in S$ with $s\neq0$, there exist $x,y\in T$ such that
\beq\label{[sxy]neq0} [s,x,y]\neq0, \eeq
\beq\label{x or y in(T:s)} x\in (T:s') ~or~ y\in (T:s'). \eeq
\end{defn}

\begin{prop}
Let $T$ be a subsystem of a LTS $S$.
\begin{enumerate}[(1)]
\item If $C(T)=0$, then $T$ is a system of quotients of itself.
\item If $S$ is a system of quotients of $T$, then $C_S(T)=C(T)=0$.
\end{enumerate}
\end{prop}
\begin{proof}
(1) Given any $s,s'\in T$ with $s\ne 0$, since $s\not\in C(T)$ there exist $x,y\in T(=(T:s'))$ such that $[s,x,y]\ne 0$.

(2) Since $C(T)\subseteq C_S(T)$  we only need to prove $C_S(T)=0$. Given any $s\in S$ there exist $x,y\in T$ such that $[s,x,y]\ne 0$. So $s\not\in C_S(T)$.
\end{proof}

The above proposition says for a LTS $T$, that $C(T)=0$ is a sufficient and necessary condition such that $T$ has a system of quotients.

We will show that some properties of a LTS $T$ can be inherited by its system of quotients $S$. Actually, $S$ just needs a weaker condition.

\begin{defn}
Let $T$ be a subsystem of a LTS $S$. Then $S$ is called a weak system of quotients of $T$, if for every $0\neq s\in S$, there exist $x,y\in T$ such that $0\neq[s,x,y]\in T$.
\end{defn}

\begin{re}
Every system of quotients of a LTS $T$ is a weak system of quotients.
\end{re}
\begin{proof}
Suppose that $S$ is a system of quotients of $T$. Then for all nonzero $s\in S$, there exist $x,y\in T$ such that $[s,x,y]\neq0$ and at least one of $x$ and
 $y$ belongs to $(T:s)$. Thus $0\neq[s,x,y]\in T$.\end{proof}

However, the converse is false. Adapted from the examples of Utumi in \cite{U} and Siles Molina in \cite{S}, we consider the set $\C[x]$ of all polynomials in
 $x$ with complex coefficients. Let $\alpha^{\sigma}$ denote the complex conjugate of a complex number $\alpha$. Then the following ternary product makes $\C[x]$
  into a LTS, denoted by $P$:
$$\left[\sum_{r=0}^{m}\alpha_rx^r, \sum_{s=0}^{n}\beta_sx^s, \sum_{t=0}^{l}\gamma_tx^t\right] = \sum_{r=0}^{m} \sum_{s=0}^{n} \sum_{t=0}^{l}
 (\alpha_r\beta_s^\sigma - \alpha_r^\sigma\beta_s)(\gamma_t + \gamma_t^\sigma)x^{r+s+t}.$$

Let $I=\langle x^4\rangle$ be the ideal of $P$ and let $S=P/I$. Let $T=\{\overline{\alpha_0}+\overline{\alpha_2x^2}+\overline{\alpha_3x^3}\}\subseteq S$.
Then $T$ is a subsystem of $S$ and $S$ is a weak system of quotients of $T$.

Indeed, $S=\overline{\C x}\dotplus T$ and $[\overline{1}, \overline{i}, \overline{1+i}]=\overline{-4i}\neq\overline{0}$. It follows that for all
 $\overline{0}\neq\overline{t}\in T$, $\overline{0}\neq[\overline{t}, \overline{i}, \overline{1+i}]=\overline{-4it}\in T$. Take
  $\overline{0}\neq\overline{\alpha_1x}+\overline{t}=
\overline{\alpha_0}+\overline{\alpha_1x}+\overline{\alpha_2x^2}+\overline{\alpha_3x^3}\in S$. If $\overline{\alpha_1x}=\overline{0}$, then
$\overline{0}\neq[\overline{t}, \overline{i}, \overline{1+i}]\in T$; if $\overline{\alpha_1x}\neq\overline{0}$, then $\overline{\alpha_1}\neq\overline{0}$,
 and hence $\overline{0}\neq[\overline{\alpha_1x}+\overline{t},\overline{i}, \overline{(1+i)x^2}]=\overline{-4i}(\overline{\alpha_0x^2}+\overline{\alpha_1x^3})\in T$.

Suppose that $S$ is a system of quotients of $T$, then for $\overline{x}, \overline{x^3}\in S$, there must be $\overline{t}, \overline{t'}\in T$ such
that $[\overline{x^3}, \overline{t}, \overline{t'}]\neq\overline{0}$ and at least one of $\overline{t}$ and $\overline{t'}$ belongs to $(T:\overline{x})$.
Since $[\overline{x^3}, \overline{t}, \overline{t'}]\neq\overline{0}$, it follows that $\overline{t}=\overline{\alpha_0}$, $\overline{t'}=\overline{\alpha_0'}$,
 and $[\overline{1}, \overline{\alpha_0}, \overline{\alpha_0'}]\neq\overline{0}$, for some $\alpha_0, \alpha_0'\in \C$. Thus
  $[\overline{x}, \overline{t}, \overline{t'}]=[\overline{1}, \overline{\alpha_0}, \overline{\alpha_0'}]\overline{x}\notin T$. But
   if $\overline{t}\in(T:\overline{x})$, then $[\overline{x}, \overline{t}, \overline{t'}]=-[\overline{t}, \overline{x}, \overline{t'}]\in T$;
   if $\overline{t'}\in(T:\overline{x})$, then $[\overline{x}, \overline{t}, \overline{t'}]=[\overline{t'}, \overline{t}, \overline{x}]-[\overline{t'},
   \overline{x}, \overline{t}]\in T$. This contradiction shows that $S$ is not a system of quotients of $T$.

\begin{re}
Let $S$ be a weak system of quotients of a LTS $T$.
\begin{enumerate}[(1)]
\item If $I$ is a nonzero ideal of $S$, then $I\cap T$ is a nonzero ideal of $T$.
\item If $T$ is semiprime (prime), then $S$ is semiprime (prime).
\end{enumerate}
\end{re}

\begin{defn}
Let $T$ be a subsystem of a LTS $S$. Then $S$ is called ideally absorbed into $T$ if for every $0\neq s\in S$, there exists an ideal $I$ of $T$ with $C_{T}(I)=0$
such that $0\neq [s,I,T]+[s,T,I]\subseteq T$.
\end{defn}

\begin{lem}\label{C(T:s)=0}
Let $T$ be a subsystem of a LTS $S$ and take $s\in S$.
\begin{enumerate}[(1)]
\item   If $S$ is a system of quotients of $T$, then $(T:s)$ is an essential ideal of $T$. Moreover, $C_{T}((T:s))=0$.
\item   If $S$ is ideally absorbed into $T$, then $(T:s)$ is an essential ideal of $T$. Moreover, $C_{T}((T:s))=0$.
\end{enumerate}
\end{lem}
\begin{proof}
(1) Let $I$ be a nonzero ideal of $T$. Take $0\neq x\in I\subseteq T$. Since $S$ is a system of quotients of $T$, there exist $y,z\in T$ such
that $[x,y,z]\neq 0$ and at least one of $y$ and $z$ belongs to $(T:s)$. By Proposition \ref{(T:s) is an ideal}, $(T:s)$ is an ideal of $T$, and
note that $I$ is an ideal of $T$, it follows that $0\neq[x,y,z]\in (T:s)\cap I$, and hence $(T:s)$ is an essential ideal of $T$.

Now, suppose $C_{T}((T:s))\neq 0$. Since $(T:s)$ is an essential ideal of $T$, there exists $0\neq t\in C_{T}((T:s))\cap (T:s)$. Note that $S$ is a
 system of quotients of $T$, then there exist $t',t''\in T$ such that $[t,t',t'']\neq 0$. However, $[t,t',t'']=0$ since $t\in C_{T}((T:s))$. This
 contradiction proves that $C_{T}((T:s))=0$.

(2) Since $S$ is ideally absorbed into $T$, there is an ideal $I$ of $T$ with $C_{T}(I)=0$ such that $0\neq [s,I,T]+[s,T,I]\subseteq T$. Then
$I\subseteq (T:s)$ by the maximum of $(T:s)$, and so $C_{T}((T:s))\subseteq C_{T}(I)=0$. Thus $(T:s)$ is an essential ideal of $T$ by Proposition
\ref{center and essential}(\ref{center implies essential}).
\end{proof}

We will show the relation between $S$ being a system of quotients of $T$ and $S$ being ideally absorbed into $T$, and the following result is needed.

\begin{lem}\label{CS(I)=0}
Let $S$ be a weak system of quotients of a LTS $T$ and let $I$ be an ideal of $T$ with $C_T(I)=0$. Then there is no nonzero $x\in S$ such that $[x,I,T]=[x,T,I]=0$.
\end{lem}
\begin{proof}
Suppose that there is $0\neq x\in S$ such that $[x,I,T]=[x,T,I]=0$. Apply that $S$ is a weak system of quotients of $T$ to find $y,z\in T$ such
 that $0\neq[x,y,z]\in T$. Since $C_T(I)=0$, $[[x,y,z],I,T]+[T,I,[x,y,z]]\neq0$. However, note that
\begin{equation*}
\begin{split}
[[x,y,z],I,T]&\subseteq[x,y,[z,I,T]]+[z,[x,y,I],T]+[z,I,[x,y,T]]\\
             &\subseteq[x,T,I]+0+[[T,I,x],T,T]+[x,[T,I,T],T]+[x,T,[T,I,T]]\\
             &\subseteq0+0+[x,I,T]+[x,T,I]=0,
\end{split}
\end{equation*}
and
\begin{equation*}
\begin{split}
[T,I,[x,y,z]]&\subseteq[[T,I,x],T,T]+[x,[T,I,T],T]+[x,T,[T,I,T]]\\
             &\subseteq0+[x,I,T]+[x,T,I]=0.
\end{split}
\end{equation*}
This is a contradiction.
\end{proof}

\begin{thm}\label{main thm}
Let $T$ be a subsystem of a LTS $S$. Then $S$ is a system of quotients of $T$ if and only if $S$ is ideally absorbed into $T$.
\end{thm}
\begin{proof}
Suppose that $S$ is a system of quotients of $T$, and take $0\neq s\in S$. By Proposition \ref{(T:s) is an ideal} and Lemma \ref{C(T:s)=0}(1), $(T:s)$ is
an ideal of $T$ with $C_T((T:s))=0$. Then $[s,(T:s),T]+ [T,(T:s),s]\neq0$, and so $[s,(T:s),T]+[s,T,(T:s)]\neq0$. Use the definition of $(T:s)$, we have
$$[s,(T:s),T]+[s,T,(T:s)]\subseteq[(T:s),T,s]+[(T:s),s,T]\subseteq T.$$
That is, $S$ is ideally absorbed into $T$.

Conversely, suppose that $S$ is ideally absorbed into $T$, then $S$ is a weak system of quotients of $T$. Take $0\neq s,s'\in S$. By Lemma \ref{C(T:s)=0}(2),
we have $C_{T}((T:s'))=0$, and hence $[s,(T:s'),T]+[s,T,(T:s')]\neq0$ by Lemma \ref{CS(I)=0}. Therefore there are $x\in(T:s')$ and $y\in T$ such that
 $[s,x,y]\neq0$ or $[s,y,x]\neq0$, which proves that $S$ is a system of quotients of $T$.
\end{proof}

In the proof of Theorem \ref{main thm}, we can conclude that $S$ is a system of quotients of $T$ if and only if $S$ is a weak system of quotients of $T$
satisfying $C_T((T:s))=0, \forall s\in S$.


We now introduce the notion of Martindale-like quotients of a LTS $T$ in \cite{GG}, and investigate the connection of which and the notion of system of quotients
 we have defined.

\begin{defn}{\rm\supercite{GG}}
A filter $\mF$ on a Lie algebra is a nonempty family of nonzero ideals such that for any $I_1, I_2 \in \mF$ there exists $I\in \mF$ such that $I\subseteq I_1\cap I_2$.
 Moreover, $\mF$ is a power filter if for any $I\in \mF$ there exists $K\in \mF$ such that $K\subseteq[I,I]$.

A Lie algebra of Martindale-like quotients $Q$ of a Lie algebra $L$ with respect to a power filter of sturdy ideals $\mF$ if $L\subseteq Q$ such that for every
nonzero element $q\in Q$ there exists an ideal $I_q\in \mF$ such that $0\neq [q,I_q]\subseteq L.$
\end{defn}

\begin{defn}{\rm\supercite{GG}}
A filter $\mF$ on a LTS is a nonempty family of nonzero ideals such that for any $I_1, I_2 \in \mF$ there exists $I\in \mF$ such that $I\subseteq I_1\cap I_2$.
Moreover, $\mF$ is a power filter if for any $I\in \mF$ there exists $K\in \mF$ such that $K\subseteq[I,T,I]$.

Let $T$ be a LTS and let $\mF$ be a filter on $T$. A LTS $S$ is a LTS of Martindale-like quotients of $T$ with respect to $\mF$ if $S$ is $\mF$-absorbed into $T$, i.e.,
 for each $0\neq s\in S$ there exists an ideal $I_s\in \mF$ such that
$$0\neq [s,I_s,T]+[s,T,I_s]\subseteq T.$$
\end{defn}

Note that if $T$ is a semiprime LTS, then the notion of ideally absorbed(i.e., the notion of system of quotients) is equivalent to the notion of Martindale quotient
over the set of all essential ideals.
We now show the relationship between a system of quotients of a LTS and an algebra of quotients of a Lie algebra.

\begin{defn}{\rm\supercite{S}}
Let $L\subseteq Q$ be an extension of Lie algebras. We say that $Q$ is an algebra of quotients of $L$ if the following equivalent conditions are satisfied:
\begin{enumerate}[(i)]
\item Given $p$ and $q$ in $Q$ with $p\neq0$, there exists $x$ in $L$ such that
    $$[x,p]\neq0 \text{ and } [x, {_L(q)}]\subseteq L,$$
    where $_L(q)$ is the linear span in $Q$ of $q$ and the elements of the form $\ad x_1\cdots \ad x_nq$ with $n\in\N$ and $x_1,\cdots,x_n\in L$.
\item For every nonzero element $q$ in $Q$ there exists an ideal $I$ of $L$ with $C_L(I)=0$ such that $0\neq[I,q]\subseteq L$.
\end{enumerate}
\end{defn}

\begin{prop}
Let $Q$ be an algebra of quotients of a Lie algebra $L$. Then $Q$ and $L$ can be regarded as LTS $\tilde{Q},\tilde{L}$ with $[a,b,c]=[[a,b],c]$, respectively.
Then $\tilde{Q}$ is a system of quotients of $\tilde{L}$.
\end{prop}
\begin{proof}
For any $0\neq q\in \tilde{Q}=Q$, since $Q$ is an algebra of quotients of a Lie algebra $L$, there exists a nonzero ideal $I$ of $L$ such that $C_L(I)=0$ and
 $0\neq[q,I]\subseteq L$. Let $\tilde{I}$ denote the induced ideal of $\tilde{L}$ by $I$ with $[a,b,c]=[[a,b],c]$. Suppose $x\in C_{\tilde{L}}(\tilde{I})$,
 then $[[x,I],L]=[x,\tilde{I},\tilde{L}]=0$, and so $[x,I]=0$. Then $x\in C_L(I)=0$, and hence $C_{\tilde{L}}(\tilde{I})=0$.

Note that $0\neq[q,I]\subseteq L$, and it follows that
$$0\neq[q,\tilde{I},\tilde{L}]=[[q,I],L]\subseteq[L,L]\subseteq L=\tilde{L}.$$
Then
$$[q,\tilde{L},\tilde{I}]\subseteq[\tilde{I},q,\tilde{L}]+[\tilde{L},\tilde{I},q]
\subseteq\tilde{L}+[[L,I],q]\subseteq\tilde{L}.$$
Therefore, $\tilde{Q}$ is a system of quotients of $\tilde{L}$.
\end{proof}

\begin{prop}\label{C(L(I))=0 iff C(I)=0}
Let $T$ be a LTS and $L(T)=T\oplus \LL(T,T)$ be its standard imbedding. Let $I$ be an ideal of $T$.
\begin{enumerate}[(1)]
\item $C_{L(T)}(I\oplus \LL(I,T))\cap T=C_T(I)$.
\item $C_{L(T)}(I\oplus \LL(I,T))=0$ if and only if $C_T(I)=0$.
\end{enumerate}
\end{prop}
\begin{proof}
(1) Suppose that $x\in C_T(I)\subseteq T$, then $[x,I,T]=[T,I,x]=0$. Note that every element $\sum\LL(a,b)\in \LL(T,T)$ with $[\sum\LL(a,b),T]=0$ implies $\sum\LL(a,b)=0$.
It follows that $\LL(x,I)=0$. The relation
$$[x,I\oplus \LL(I,T)]=\LL(x,I)\oplus[I,T,x]=0$$
shows that $x\in C_{L(T)}(I\oplus \LL(I,T))\cap T$.

Conversely, let $x\in C_{L(T)}(I\oplus \LL(I,T))\cap T$. Then $[x,I\oplus \LL(I,T)]=0$, and so $\LL(x,I)=0, [I,T,x]=0$. Thus $[x,I,T]=0$, it follows that $x\in C_T(I)$.

(2) Only sufficiency needs proof. If $C_T(I)=0$ let us prove that $C_{L(T)}(I\oplus \LL(I,T))=0$. Note that $C_{L(T)}(I\oplus \LL(I,T))\cap T=0$. Then for any $\LL(y,z)\in C_{L(T)}(I\oplus \LL(I,T))$,
 the equation $[\LL(y,z), I\oplus \LL(I,T)]=0$ implies that $[y,z,I]=0$ and $[\LL(y,z),\LL(I,T)]=0$. Thus, $\LL([y,z,T],I)\subseteq [\LL(y,z),\LL(T,I)]+\LL([y,z,I],T)=0,$ and so $[[y,z,T],I,T]=0$.
  Notice that
\begin{equation*}
\begin{split}
[T,I,[y,z,T]]\subseteq &[y,z,[T,I,T]]+[[y,z,T],I,T]+[T,[y,z,I],T]\\
             \subseteq &[y,z,I]+0+0=0.
\end{split}
\end{equation*}
Then $[y,z,T]\subseteq C_T(I)=0$, and this proves $\LL(y,z)=0$. Therefore, $C_{L(T)}(I\oplus \LL(I,T))=0$.
\end{proof}


\begin{prop}{\rm\supercite{GG}}\label{S/T-F iff L(S)/L(T)-L(F)}
Let $T$ be a subsystem of a LTS $S$ and let $L(T)$ and $L(S)$ be standard embeddings of $T$ and $S$, respectively. Then $S$ is a LTS of Martindale-like
quotients of $T$ with respect to a power filter $\mF$ of ideals with zero centralizer on $T$ if and only if $L(S)$ is a Lie algebra of quotients of $L(T)$
with respect to the power filter $L(\mF)=\{I\oplus \LL(I,T)| I\in \mF\}$ of ideals with zero centralizer on $L(T)$.
\end{prop}

\begin{prop}{\rm\supercite{GGN}}\label{T is nondege. iff L(T) is nondege.}
A LTS $T$ is nondegenerate if and only if $L(T)$ is a nondegenerate
Lie algebra.
\end{prop}

\begin{prop}{\rm\supercite{S}}\label{L nondege. implies Q nondege.}
If $Q$ is an algebra of quotients of a nondegenerate Lie algebra $L$ and the characteristic of the base field is not 2 or 3, then $Q$ is nondegenerate.
\end{prop}

\begin{thm}\label{S/T implies L(S)/L(T)}
Let $S$ be a system of quotients of a semiprime LTS $T$. Then $L(S)$ is an algebra of quotients of $L(T)$. Moreover, if $T$ is nondegenerate and $\ch{\bf F}\neq2,3$,
 then $S$ is also nondegenerate.
\end{thm}
\begin{proof}
Since $T$ is semiprime, the set $\mF$ of all ideals of $T$ with zero centralizer is a power filter on $T$ such that $S$ is a LTS of Martindale-like quotients of $T$
with respect to $\mF$. It follows from Proposition \ref{S/T-F iff L(S)/L(T)-L(F)} that $L(S)$ is a LTS of Martindale-like quotients of $L(T)$ with respect to $L(\mF)$.
 Then for each $0\neq x\in L(S)$, there exists an ideal $I\oplus \LL(I,T)\in L(\mF)$ such that $0\neq [x, I\oplus \LL(I,T)]\subseteq L(T)$. Note that $C_{L(T)}(I\oplus \LL(I,T))=0$, since
  $C_T(I)=0$ and by Proposition \ref{C(L(I))=0 iff C(I)=0}. Therefore $L(S)$ is an algebra of quotients of $L(T)$.

Suppose that $T$ is nondegenerate. Then by Proposition \ref{T is nondege. iff L(T) is nondege.}, $L(T)$ is nondegenerate, and so $L(S)$ is nondegenerate by
 Proposition \ref{L nondege. implies Q nondege.}. Therefore, $S$ is nondegenerate again by Proposition \ref{T is nondege. iff L(T) is nondege.}.
\end{proof}

At the end of this paper, we will construct the maximal system of quotients for nondegenerate Lie triple systems.

In \cite{GGN}, the authors built the maximal Lie algebra $Q$ of a Lie algebra $L$ with respect to a filter $\mF$ and showed that if $L$ is
a nondegenerate Lie algebra with finite $\Z$-grading and $\mF$ is a power filter of ideals with zero centralizer on $L$, then $Q$ has a finite $\Z$-grading.
 Moreover, $L$ and $Q$ have the same support. It is also proved that if $S$ is another Lie algebra of $L$ with respect to $\mF$, then there is a Lie monomorphism
 of $S$ into $Q$ which is the identity on $L$.

Let $T$ be a nondegenerate LTS and $\mF$ be the power filter of $T$ consisting of all ideals with zero centralizer. Then $L(T)$ is a nondegenerate Lie algebra by
 Proposition \ref{T is nondege. iff L(T) is nondege.} and $L(\mF)$ is the power filter of $L(T)$ consisting of ideals with zero centralizer by Proposition
  \ref{S/T-F iff L(S)/L(T)-L(F)}. Suppose that the $\Z_2$-graded Lie algebra $Q=Q_0\oplus Q_1$ is the maximal Lie algebra of quotients of $L(T)$ with respect
  to $L(\mF)$. Remember that $Q_1$ is a LTS by putting $[q,q',q'']=[[q,q'],q''], \forall q,q',q''\in Q_1$.

For any nonzero $q\in Q_1$, there exists $I\oplus \LL(I,T)\in L(\mF)$ such that $0\neq[q,I\oplus \LL(I,T)]\subseteq T\oplus \LL(T,T)$. Then $\LL(q,I)\subseteq \LL(T,T)$ and $[T,I,q]\subseteq T$,
and so $[q,I,T]\subseteq [T,T,T]\subseteq T$, and it follows that $[q,I,T]+[q,T,I]\subseteq T$. Since $C_{L(T)}(I\oplus \LL(I,T))=0$, $C_T(I)=0$ by Proposition \ref{C(L(I))=0
iff C(I)=0}, which implies $[q,I,T]+[q,T,I]\neq0$. Therefore, $Q_1$ is a system of quotients of $T$.

Moreover, $Q_1$ is maximal in the sense that if $S$ is another system of quotients of $T$, then there exists a Lie monomorphism $\phi: S\rightarrow Q_1$ which is
the identity on $T$. Now assume that $S$ is a system of quotients of $T$. Then $L(S)$ is an algebra of quotients of $L(T)$ by Proposition \ref{S/T implies L(S)/L(T)}.
 Note that $Q=Q_0\oplus Q_1$ is the maximal Lie algebra of quotients of $L(T)$ with respect to $L(\mF)$. Then there exists a Lie monomorphism $\phi: L(S)\rightarrow Q$
 which is the identity on $L(T)$, and hence $\phi|_S: S\rightarrow Q_1$ is a Lie monomorphism which is the identity on $T$. This implies that $Q_1$ is the maximal
 system of quotients of $T$.

\begin{lem}{\rm\supercite{S}}\label{semisimple Lie alg}
If $L$ is a finite dimensional semisimple Lie algebra over an algebraically closed field of characteristic 0, i.e., the solvable radical of $L$ is zero, then the maximal Lie algebra of quotients of $L$
 is $L$.
\end{lem}

\begin{prop}
Let $T$ be a finite dimensional semisimple LTS over an algebraically closed field of characteristic 0 and $S$ be the maximal system of quotients of $T$. Then $S=T$.
\end{prop}
\begin{proof}
Since $T$ is semisimple, its solvable radical $R(T)=0$, then $R(L(T))=R(T)\oplus\LL(R(T),T)=0$(cf. \cite{L}), which implies that $L(T)$ is semisimple. Hence the maximal
 Lie algebra of quotients of $L(T)$ is itself by Lemma \ref{semisimple Lie alg}, then the maximal system of quotients of $T$ is $T$.
\end{proof}

 {\bf ACKNOWLEDGMENTS}\quad The authors would like to thank the referee for valuable comments and
suggestions on this article.

  The work was supported by  NNSF of China (No. 11171055, No. 11226054), NSF of Jilin Province (No. 201115006), Scientific Research Foundation for
Returned Scholars Ministry of Education of China and the Fundamental
Research Funds for the Central Universities(No. 12SSXT139).

\end{document}